\documentclass[11pt]{amsart}
\usepackage {amssymb}
\usepackage{verbatim}
\usepackage{color, xcolor,soul}

\input epsf

\newlength{\tabwidth}
\newlength{\tabheight}
\newlength{\tabrule}
\newlength{\tabwidthx}
\newlength{\tabheightx}

\def\gentabbox#1#2#3#4{\vbox to \tabheight{\setlength{\tabrule}{#3}%
  \setlength{\tabwidthx}{#1\tabwidth}\addtolength{\tabwidthx}{\tabrule}%

\setlength{\tabheightx}{#2\tabheight}\addtolength{\tabheightx}{-\tabheight}%
  \hbox to #1\tabwidth{%
    \hspace{-0.5\tabrule}\rule{\tabrule}{#2\tabheight}\hspace{-\tabrule}%
    \vbox to #2\tabheight{\hsize=\tabwidthx%
      \vspace{-0.5\tabrule}\hrule width\tabwidthx height\tabrule%
      \vspace{-0.5\tabrule}\vfil%
      \hbox to \tabwidthx{\hss#4\hss}%
        \vfil\vspace{-0.5\tabrule}%
      \hrule width\tabwidthx height\tabrule\vspace{-0.5\tabrule}}%
    \hspace{-\tabrule}\rule{\tabrule}{#2\tabheight}\hspace{-0.5\tabrule}}%
  \vspace{-\tabheightx}}}
\def\genblankbox#1#2{\vbox to \tabheight{\vfil\hbox to
#1\tabwidth{\hfil}}}
\def\tabbox#1#2#3{\gentabbox{#1}{#2}{0.4pt}{\strut #3}}

\catcode`\:=13 \catcode`\.=13 \catcode`\;=13 
\catcode`\>=13 \catcode`\^=13
\def:#1\\{\hbox{$#1$}}
\def.#1{\tabbox{1}{1}{$#1$}}
\def>#1{\tabbox{2}{1}{$#1$}}
\def^#1{\tabbox{1}{2}{$#1$}}
\def;{\genblankbox{1}{1}\relax}
\catcode`\:=12 \catcode`\.=12 \catcode`\;=12 
\catcode`\>=12 \catcode`\^=7

\newcommand{\field}{\mathbb}

\newcommand{\C}{{\field C}}
\newcommand{\R}{{\field R}}

\newcommand{\lam}{\lambda}

\newcommand{\vareps}{\varepsilon}

\newcommand{\AV}{\mathrm{AV}}
\renewcommand{\big}{{\dagger}}

\newcommand{\Ind}{\mathrm{Ind}}
\newcommand{\AC}{\caA\caV}

\newcommand{\ol}{\overline}

\newcommand{\lra}{\longrightarrow}

\newcommand{\Aql}{A_{\frq}(\lam)}

\newcommand{\Sp}{\mathrm{Sp}}
\newcommand{\SO}{\mathrm{SO}}

\newcommand{\U}{\mathrm{U}}

\newcommand{\Unip}{\mathrm{Unip}}
\newcommand{\GL}{\mathrm{GL}}

\newtheorem{prop}{Proposition}[subsection]

\newtheorem{lemma}[prop]{Lemma}
\newtheorem{theorem}[prop]{Theorem}

\newtheorem{conj}[prop]{Conjecture}

\newtheorem{conjecture}[prop]{Conjecture}

\theoremstyle{definition}

\newtheorem{remark}[prop]{Remark}
\newtheorem{example}[prop]{Example}

\newtheorem{definition}[prop]{Definition}

\newcommand{\fra}{\mathfrak{a}}

\newcommand{\frg}{\mathfrak{g}}
\newcommand{\frh}{\mathfrak{h}}

\newcommand{\frk}{\mathfrak{k}}
\newcommand{\frl}{\mathfrak{l}}

\newcommand{\frn}{\mathfrak{n}}
\newcommand{\fro}{\mathfrak{o}}
\newcommand{\frp}{\mathfrak{p}}
\newcommand{\frq}{\mathfrak{q}}

\newcommand{\frs}{\mathfrak{s}}
\newcommand{\frt}{\mathfrak{t}}
\newcommand{\fru}{\mathfrak{u}}

\newcommand{\bbC}{\mathbb{C}}

\newcommand{\bbH}{\mathbb{H}}

\newcommand{\bbR}{\mathbb{R}}

\newcommand{\caA}{\mathcal{A}} 

\newcommand{\caC}{\mathcal{C}}

\newcommand{\caN}{\mathcal{N}}
\newcommand{\caO}{\mathcal{O}}

\newcommand{\caV}{\mathcal{V}}

\begin{document}

\title[unipotent representations]{unitarity of unipotent representations of $\Sp(p,q)$ and $\SO^*(2n)$}

\author{Dan M. Barbasch, Peter E. Trapa}

\address{Department of Mathematics, Cornell University, Ithaca, NY 14853}
\email{dmb14@cornell.edu}

\address{Department of Mathematics, University of Utah, Salt Lake City, 
UT 84112}
\email{ptrapa@math.utah.edu}

\thanks{DB was partially supported by NSA grant H98230-16-1-0006.}

\thanks{PT was partially supported by NSF DMS-1302237.}

\begin{abstract}
The purpose of this paper is to define a set of  representations of $\Sp(p,q)$ and $\SO^*(2n)$, the unipotent 
representations of the title, and establish their unitarity.  The unipotent representations considered
here properly contain the special unipotent representations of Arthur and Barbasch-Vogan; in particular
we settle the unitarity of  special unipotent representations for these groups.  
\end{abstract}

\maketitle

\section{introduction}
\label{s:intro}

A long-standing heuristic in the unitary representation theory of a real reductive group $G_\R$
is that the smallest representations with a fixed interesting infinitesimal character
should be unitary.  When the notions of ``interesting'' and ``smallest''
can be made precise,
the corresponding representations are often
fundamental building blocks of the entire unitary dual, and are called 
unipotent.  Because of the algebraic
nature of this kind of approach, it is typically challenging to prove that unipotent
representations defined in this way are indeed unitary.

The purpose of this paper is to prove that certain unipotent representations of the
non quasi-split groups $\Sp(p,q)$ and $\SO^*(2n)$ are unitary.  The
 representations we treat should suffice as building blocks
 for the integral unitary dual (cf.~Conjecture \ref{c:intro} below).

In particular, one class of unipotent representations
that we treat, the special unipotent
representations, originate in conjectures of Arthur and are predicted
to be local component of automorphic forms (which of course would give
an explanation of their unitarity).
In the case of of a real group $G_\R$, 
the ideas of Arthur were made precise and refined by
Barbasch-Vogan \cite{bv2} and most completely in \cite[Chapter 26]{abv}. 
Our first main result establishes
  that the special unipotent representations of 
  $\Sp(p,q)$ and $\SO^*(2n)$ are unitary.

We begin by recalling the definition of special unipotent
representations in general.  Suppose $G_\R$ is the real points of a
connected reductive complex algebraic group $G$.  
Write $G^\vee$ for the Langlands
dual of $G$, and write $\frg$ and $\frg^\vee$
for the respective Lie algebras.  The construction of the dual group
specifies Cartan subalgebras $\frh$ and $\frh^\vee$ (of $\frg$ and 
$\frg^\vee$), an isomorphism between them,
\begin{equation}
\label{e:hisom}
\frh^\vee \stackrel{\sim}{\rightarrow} \frh^*,
\end{equation}
and an isomorphism between $W = W(\frg,\frh)$ and $W^\vee = W(\frg^\vee, \frh^\vee)$.

Fix a nilpotent adjoint orbit 
\begin{equation}
\label{e:O}
\caO^\vee \subset \frg^\vee.  
\end{equation}
Using the
Jacobson-Morozov Theorem,
choose
an $\frs\frl_2$ triple $\{e^\vee, f^\vee, h^\vee\}$ with $h^\vee \in \frh^\vee$
and $e^\vee \in \caO^\vee$.  Although $h^\vee$ depends on this choice,
its $W^\vee$ orbit does not.  Using \eqref{e:hisom}, we obtain a well-defined element
\begin{equation}
\label{e:lam}
\chi(\caO^\vee) = \frac12 h^\vee \in \frh^\vee/W^\vee \simeq \frh^*/W.
\end{equation}
According to the Harish-Chandra isomorphism, $\chi(\caO^\vee)$ specifies
a maximal ideal 
\begin{equation}
\label{e:z}
Z(\caO^\vee) \triangleleft Z(\frg),
\end{equation}
in the center $Z(\frg)$ of the enveloping algebra $U(\frg)$ of $\frg$.
It is easy to see that there is a unique primitive ideal 
\begin{equation}
\label{e:j}
I(\caO^\vee) \triangleleft U(\frg)
\end{equation}
containing $Z(\caO^\vee)$ which is maximal in the inclusion order
on primitive ideals.   Representations annihilated by the maximal $I(\caO^\vee)$ are 
therefore small and have the interesting infinitesimal
character $\chi(\caO^\vee)$.

\begin{definition}
\label{d:su}
Let $X$ be an irreducible $(\frg,K)$ module. 
The $X$ is said to be special unipotent if there
exists $\caO^\vee \subset \frg^\vee$ as above such that 
\[
\mathrm{Ann}_{U(\frg)}(X) =I(\caO^\vee).
\]
\end{definition}

\begin{conjecture}[Arthur, Barbasch-Vogan]
\label{c:bv}
Suppose $X$ is a special unipotent $(\frg,K)$ module for $G_\R$.  Then
$X$ is unitarizable.
\end{conjecture}

Our first main result is: 

\begin{theorem}
\label{t:intro1}
Conjecture \ref{c:bv} holds for $G_\R = \Sp(p,q)$ and $\SO^*(2n)$.
\end{theorem}

\medskip

For $\Sp(p,q)$ and $\SO^*(2n)$, it turns out that there is essentially only one other kind of interesting infinitesimal character.  These infinitesimal 
characters
are attached to nilpotent orbits
in $\frg$ (not $\frg^\vee)$: given such an orbit $\caO$, we let $\chi'(\caO)$ denote the corresponding infinitesimal character.  The definition is somewhat  ad hoc, and is given in Sections \ref{s:sppq-ic} and \ref{s:sostar-ic}; it is also given in \cite{msz}
(under some mild restrictions).  Unlike $\chi(\caO^\vee)$,
$\chi'(\caO)$ is always integral.  We let $\Unip'(\caO)$ denote the set of irreducible representations with infinitesimal character
$\chi'(\caO)$ annihilated by the maximal primitive ideal $I'(\caO)$ containing the maximal ideal $Z'(\caO)$ in $Z(\frg)$ corresponding
to $\chi'(\caO)$; see Section \ref{s:sppq-counting}
and \ref{s:sostar-counting}.

Our second main result is:

\begin{theorem}
\label{t:intro2}
Every representation in $\Unip'(\caO)$ for  $G_\R = \Sp(p,q)$ and $\SO^*(2n)$ is unitary.
\end{theorem}

The two results have some overlap: if $d$ denotes Spaltenstein duality
as in \cite[Appendix]{bv2},  
\[
\chi'(\caO) = \chi(d(\caO))
\]
if and only if $d(\caO)$ is even, i.e.~if and only if $\chi(d(\caO))$ is integral; see Sections \ref{s:sppq-ic} and \ref{s:sostar-ic}.  Thus, Theorems \ref{t:intro1} and \ref{t:intro2} overlap exactly
in the integral special unipotent representations.  In this special case, another proof of unitarity using the theta correspondence has recently been given in \cite{msz}.  The other representations appearing in Theorems \ref{t:intro1}  and  \ref{t:intro2} do not appear
to be accessible by the techniques of the theta correspondence (but it should be noted that \cite{msz} also handles other classical groups which are much more complicated).

The proof of Theorem \ref{t:intro2} that we give below for $\Sp(p,q)$ in Sections \ref{s:intro2-proof} is conceptually very simple, and is essentially self-contained
apart from relying only on  Vogan's
unitarity results for weakly fair $\Aql$ modules (as recalled in Section \ref{s:coh-ind}) and elementary calculation.  We construct some unipotent representations
as $\Aql$ modules in Sections \ref{s:sppq-aql} and \ref{s:sostar-aql}, and therefore establish their unitarity.  For a unipotent representation $\pi$ which is not covered
by this construction, we use parabolic induction to induce $\pi$ to an irreducible representation $\pi^\dagger$ of a larger group.  The induction is engineered to
preserved both unitarity {\em and} nonunitary (Lemma \ref{l:sppq-main}).  We then recognize the representation $\pi^\dagger$ for the larger group as one of the unipotent
$\Aql$ whose unitarity we have already established; hence $\pi$ must also be unitarity, completing the proof.  
Section \ref{s:intro1-proof} contains details of the 
nonintegral special unipotent representations in Theorem \ref{t:intro1} not covered 
by Theorem \ref{t:intro2}

The details for $\SO^*(2n)$ are entirely similar.  They are briefly given in Section \ref{s:sostar}.

To conclude, we indicate one sense in which the unipotent representations studied here should serve as building blocks for the full unitary dual.  The following result is stated for $\Sp(p,q)$,
but the obvious version applies to $\SO^*(2n)$ as well.

\begin{conj}
\label{c:intro}
Suppose $\pi$ is a unitary representation of $G_\R = \Sp(p,q)$ with integral infinitesimal character.  Then there exists
\begin{enumerate}
\item[(a)] A $\theta$-stable parabolic subalgebra $\frq = \frl \oplus \fru$ such that the normalizer $L_\R$ of $\frl$ in $G_\R$ is isomorphic to 
\[
\Sp(p_0,q_0) \times \U(p_1,q_1) \times \cdots \times \U(p_r,q_r);
\]

\item[(b)] A representation $\pi_0 \in \Unip'(\caO_0)$ for some nilpotent orbit for $\Sp(p_0,q_0)$; and

\item[(c)]  Weakly fair $\Aql$ modules $\pi_i$ for $\U(p_i,q_i)$, $1\le i \le r$,
\end{enumerate}
such that $\pi$ is cohomologically induced from $\pi_0 \boxtimes \pi_1 \boxtimes \cdots \boxtimes \pi_r$ in the weakly fair range.
\end{conj}

Theorem \ref{t:intro2} (and Vogan's results on cohomological induction in the weakly fair range \cite{v:unit}) imply that
all the representations appearing in the conjecture are unitary; the open question is whether these representations
are exhaustive.

\section{background}
\label{s:back}
\subsection{General notation}
\label{s:gen-not}
Let $G_\bbR$ denote the real points of a connected reductive algebraic group $G$ with maximal 
compact subgroup $K_\bbR$ corresponding to a Cartan involution $\theta$.  Write $\frg = \frk \oplus
\frp$ for the complexified Cartan decomposition.  
When convenient, we will identify $(\frg/\frk)^*$ with $\frp$.  

We will let  $\frh = \frt \oplus \fra$ denote a maximally compact Cartan
subalgebra of $\frg$, and $\Delta^+$ a choice of positive roots of $\frh$ in $\frg$.

\subsection{Nilpotent orbits}
\label{s:nilp}
According to the Sekiguchi correspondence (e.g.~\cite[Chapter 9]{cm}), there is a canonical bijection between the nilpotent $K$ orbit on $\frp$
and the $G_\R$ orbits on $\frg_\R$.  Their number is finite.

\subsection{Associated varieties and cycles}
\label{s:av}
If $X$ is a finitely generated $(\frg,K)$ module, then we let $\AV(X)$ denote
its associated variety as defined in \cite{v:av}.  
In particular, we may
write
\[
\AV(X) = \ol{\caO_K^1} \cup \cdots \cup \ol{\caO_K^l}
\]
for orbits $\caO_K^i$ of $K$ on the nilpotent elements in $(\frg/\frk)^* \simeq \frp$.  If we further
assume $X$ is irreducible, \cite{v:av} establishes that
\[
\caO := G\cdot \caO_K^i 
\]
is well-defined independent of the choice of $i$ and is dense in the associated variety
of the annihilator $I_X$ of $X$ (i.e.~the variety in $\frg^*$ cut out by $\mathrm{gr}(I_X) \triangleleft S(\frg)$).
As in \cite{v:av}, the construction of the associated variety
may be refined to produce a positive integral linear combination
\[
\AC(X) = \sum_i m_i {\caO_K^i},
\]
called the associated cycle of $X$.

\subsection{Associated varieties and special unipotent representations}
\label{s:av-unip}
The following is a useful criterion for determining if a representation is special
unipotent.

\begin{prop}
\label{p:av-unip}
Recall the notation of the introduction.  Fix $\caO^\vee$, set $\caO = d(\caO^\vee)$,
and suppose $\caO_K$ is a $K$ orbit 
such that $G \cdot \caO_K = \caO$.  Suppose $\pi$ is an irreducible
$(\frg,K)$ module with infinitesimal character $\chi(\caO^\vee)$
such that $\caO_K$ is dense in an irreducible component of $AV(\pi)$.
(For example, suppose $AV(\pi) = \ol{\caO_K}$.)
Then $\pi$ is special unipotent attached to $\caO^\vee$.  
\end{prop}
\begin{proof}
The hypothesis on $\pi$ imply that the associated variety of its annihilator $I_\pi$ equals $\ol{\caO}$.  
Meanwhile, by the appendix to \cite{bv2}, the associated variety of the primitive ideal $I(\caO^\vee)$
is also $\ol{\caO}$.  Since $I_\pi$ contains $Z(\caO)$, and since $I(\caO^\vee)$ is maximal with this property,
$I_\pi \subset I(\caO^\vee)$.   Since they also have the same associate variety, \cite[Korollar
3.4 and Satz 7.1]{borho-jantzen} show that they are equal.  Hence $\pi$ has annihlated $I(\caO^\vee)$
and is therefore special unipotent.
\end{proof}

\subsection{Associated varieties and real parabolic induction}
\label{s:real-ind}
Let $P_\R = L_\R N_\R$ be a real parabolic subgroup of $G_\R$.  Let
$\pi$ be an irreducible representation of $L_\R$ and extended trivially to $P_\R$.  
Then we can form the (normalized) induced representation
$\Ind_{P_\R}^{G_\R}(\pi).$
We will need to record how associated varieties behave with respect to real induction  \cite{babo}.

\begin{prop}
\label{p:av-real-ind}
In the setting above, write
\[
\AV(\pi) = \ol{\caO_{L\cap K}^1} \cup \cdots \cup  \ol{\caO_{L\cap K}^l} 
\]
and write $\caO_{\R,L}^i$ for the nilpotent $L_\R$ on $\frl_\R$ corresponding to $\caO_{L \cap K}^i$
via the Sekiguchi correspondence.
Then
\[
G_\R \cdot \left (
\bigcup_i \ol{\caO_{\R,L}^i} + \frn_\R
\right )
\]
is a union of closures of equidimensional $G_\R$ orbits on $\frg_\R$ which we may write as
\[
\ol{\caO_{\bbR^1}} \cup \cdots \cup  \ol{\caO_{\bbR}^\ell}.
\]
Let $\caO_K^i$ denote the $K$ orbit on $\frp$ corresponding to $\caO_\R^i$ via the
Sekiguchi correspondence.
Then
\[
\AV
\left (
\Ind_{P_\R}^{G_\R}(\pi)
\right )
= 
\bigcup_{i=1}^\ell \ol{\caO_K^i}.
\]
\end{prop}

This cumbersome formulation is meant to avoid any discussion of asymptotic supports and the Barbasch-Vogan conjecture.  A self-contained
proof of Proposition \ref{p:av-real-ind} is given in  \cite{babo} (without recourse to the proof of the Barbasch-Vogan conjecture).

\subsection{Cohomological induction}
\label{s:coh-ind}
Let $G_\bbR$ be as in Section \ref{s:gen-not}.  Fix a $\theta$-stable parabolic subalgebra $\frq = \frl \oplus \fru$
such that $\frh \subset \frl$ and the roots of $\frh$ in $\fru$ are contained in $\Delta^+$.  Let $\rho(\fru)$ denote
the half-sum of the roots of $\frh$ in $\fru$,
and write $L$ for the normalizer of $\frl$ in $G$.

Let $\bbC_\lam$ denote a one-dimensional $(\frl, L \cap K)$ module with differential $\lam \in \frh^*$.
Write $A_\frq(\lam)$ for the $(\frg,K)$ module cohomologically induced as in \cite[Chapter V]{kv}.
Here are the some properties of these modules that we will need.

\begin{theorem}
\label{t:aql}
In the setting of the previous paragraph, 
\begin{enumerate}
\item[(i)] The infinitesimal character of $A_\frq(\lam)$ is represented by $\lam  +\rho$.

\item[(ii)] Suppose that $\lam$ is in the good range for $\frq$, 
\[
\langle \lam + \rho, \alpha ^\vee \rangle \geq 0 \text{ for all $\alpha \in \Delta(\fru)$.}
\]
Then $A_\frq(\lambda)$ is nonzero, irreducible, and unitary.

\item[(iii)]
Suppose that $\lam$ is in the weakly fair range for $\frq$, 
\[
\langle \lam + \rho(\fru), \alpha ^\vee \rangle \geq 0 \text{ for all $\alpha \in \Delta(\fru)$.}
\]
Then $A_\frq(\lambda)$ is unitary (but possibly irreducible or zero) with associated variety
\[
\AV(A_\frq(\lambda)) = K \cdot (\fru \cap \frp).
\]

\item[(iv)]  In the setting of (iii), suppose in addition that 
\[
\dim
\left(
G\cdot (\fru \cap \frp) 
\right )
= 
\dim 
\left (
G \cdot \fru
\right ).
\]
Then $\Aql$ vanishes only if there is a $\theta$-stable parabolic $\frq' = \frl'  \oplus \fru'$ containing $\frq$ such that $L'/L$ is compact and $\lam+\rho$ is singular for a root of $\frh$ in $\frl'$.  Otherwise
$\Aql$ is nonzero and irreducible.  Moreover the multiplicity of the dense $K$ orbit $\caO_K$ in $K \cdot (\fru \cap \frp)$ in the
associated cycle of $\Aql$ divides the order of the component group of the centralizer in $K$ of a point of $\caO_K$.
\end{enumerate}
\end{theorem}
\begin{proof}[Sketch.]
Part (i) is built into the normalization we adopt from \cite{kv}.  Parts (ii) and (iii), apart from the associated variety statements, are proved in \cite{v:unit}; see
\cite[Chapter VIII]{kv}.  The irreducibility
and nonvanishing assertions in (iv) are proved in \cite[Theorem 6.5]{v:ds}.
The multiplicity of the associated cycle in (iv) (under the extra hypothesis in (iv)) counts the number of points in a single orbits of the centralizer
in $K$ of a point of $\caO_K$; see \cite[Proposition 3.12]{t:ltc}.  Hence the assertion about multiplicities in (iv) follows.
\end{proof}

\section{details for $\Sp(p,q)$}
\label{s:sppq}
In this section we restrict to $G_\R = \Sp(p,q)$ and use the notation of Section \ref{s:gen-not} in this setting.  In particular
$G=\Sp(2n,\bbC)$ with $n=p+q$.  We identify $\frh^*$ with $\bbC^n$ using the standard coordinates.

\subsection{Nilpotent orbits}
\label{s:sppq-nilp}
The nilpotent orbits of $G$ on $\frg$ are parametrized by partitions of $2n$ in which all odd parts occur with even multiplicity.
The correspondence is given by taking Jordan form.

The nilpotent $K$ orbits on $\frp$ for $\Sp(p,q)$ are parametrized by certain signed tableau of signature $(2p,2q)$
(\cite[Chapter 9]{cm}).  These are Young diagrams of size $2n$ in which the boxes are filled with $2p$ plus signs and $2q$ minus
 signs alternating across rows, modulo the equivalence of interchanging rows of equal length, satisfying certain restrictions described below.
To each (equivalence class of) signed tableau, we can attached a string 
\[
(m_1)_+^{k_{m_1}^+}(m_1)_-^{k_{m_1}^-}(m_2)_+^{k_{m_2}^+}(m_2)_-^{k_{m_2}^-}\dots
\]
with $m_1>m_2>\cdots$ indicating that there are $k_{m_i}^+$ rows of length $m_i$ beginning with $+$, and similarly for $k_{m_i}^-$.  The restrictions are
 \begin{enumerate}
 \item[(i)] For every odd $m$, $k_m^\vareps$ is even; that is, all odd parts beginning with a fixed sign occur an even number of times;
  \item[(i)] For each even part $m$, $k_m^+ = k_m^-$ is even; that is,  the number of rows of a fixed even length beginning with $+$
  equals the number rows of that length beginning with $-$.
  \end{enumerate}
In particular, all parts occur with even multiplicity (or, equivalently, all columns have even length).  

This parametrization is arranged
so that the shape of signed tableau parametrizing a $K$ orbit on $\frp$ is its Jordan form.  So a complex nilpotent orbit
$\caO$ meets $\frp$ if and only if there is an arrangement of signs in its Jordan form satisfying the above conditions.

\subsection{Infinitesimal characters}
\label{s:sppq-ic}
Fix a nilpotent orbit $\caO$ for $G$, and consider the partition of $2n$ corresponding to its Jordan form.  
Let $c_1\ge c_2 \ge \cdots$ denote the transposed partition.  As remarked in Section \ref{s:sppq-nilp}, a 
necessary condition for $\caO$ to meet $\frp$ is that all
$c_i$ are even, so write $c_i =2d_i$.  To each $d_i$ associate a string of nonnegative integers,
\[
\text{$(1,\dots,d_i)$  if $i$ is odd and  $(0,\dots,d_{i-1})$ if $i$ is even}.
\]
Let $\chi'(\caO)$ denote the concatenation of these strings, and interpret $\chi'(\caO)$ as an infintesimal character for
$\Sp(2n,\bbC)$ in the standard coordinates.  For example, if $\caO$ is zero with corrresponding partition $1+\cdots +1=2n$,
then $c_1 =2n$, and $\chi'(\caO) = (1,2,\dots,n)$, i.e.~$\rho$.  If $\caO$ corresponds to the partition $n+n=2n$, then
$c_1=\cdots=c_n =2$, and $\chi'(\caO)=(1,0,1,0,\dots,1,0). $

Let $Z'(\caO)$ denote the maximal ideal in $Z(\frg)$ corresponding to the $W$ orbit of $\chi'(\caO)$ under the Harish-Chandra
isomorphism, and let $I'(\caO)$ denote the unique maximal primitive ideal containing $Z'(\caO)$.  

\begin{prop}
\label{p:sppq-pi}
With the notation just introduced, $I'(\caO)$ is
the unique primitive ideal containing $Z'(\caO)$ whose associated variety is the closure of $\caO$.
\end{prop}
\begin{proof}
We first claim that there is a unique primitive ideal, say $I$, containing $Z'(\caO)$ 
and associated variety equal to the closure of $\caO$.  One way to see
this is to consult the classification
of primitive ideals in terms of domino tableaux \cite{bv:class}: one checks easily that there is a unique domino tableau with shape
given by the the Jordan form of $\caO$ whose dominos are labeled by the coordinates of $\chi'(\caO)$ such that the labels
weakly decrease across rows and strictly decrease down columns. 

Next we claim that this primitive ideal is maximal in the inclusion partial order on primitive ideals.  (This will imply 
$I = I'(\caO)$, and finish the proof of the proposition.)
If there is another primitive ideal $J \supsetneq I$, then \cite[Korollar
3.4 and Satz 7.1]{borho-jantzen} imply that
the associated variety of $J$ is properly contained in the associated variety of $I$.  So in the classification of primitive
ideals by domino tableau, $J$ is parametrized by a tableau whose shape is special and strictly smaller than $I$ (in the partial order
on tableau) and whose whose dominos are labeled by the coordinates of $\chi'(\caO)$ such that the labels
weakly decrease across rows and strictly decrease down columns.   One quickly checks that no such domino tableaux exist.
So $I$ is indeed maximal, and $I=I'(\caO)$. 

\end{proof}


\bigskip
Recall the Spaltenstein dual orbit $d(\caO)$ for $\SO(2n+1,\C)$ and the infinitesimal character $\chi(d(\caO))$
discussed in the introduction.

\begin{prop}
\label{p:sppq-ic}
Fix a complex nilpotent orbit $\caO$ for $\Sp(2n,\bbC)$ meeting $\frp$ for $\Sp(p,q)$.  Then
\[
\chi(d(\caO)) = \chi'(\caO) \text{ as infinitesimal characters if and only if  $\caO^\vee=d(\caO)$ is even}.
\]
If we write $k$ for the largest part of the Jordan form corresponding  to $\caO$,
 $\caO^\vee$ is even iff every even part less than or equal to $k$ actually appears with nonzero multiplicity
  in the Jordan form of $\caO$.
\end{prop}

\begin{proof}
To begin, we recall the the procedure for computing $d(\caO)$.  Let $\nu$ denote the
partition of $2n$ parametrizing $\caO$, and recall that
nilpotent orbits $\frs\fro(2n+1,\bbC)$ are parametrized by partitions of $2n+1$
in which even parts occur with even multiplicity. (Such a partition is called a B-partition.)
To compute $d(\caO)$, first
add 1 to the largest part of $\nu$ and let $\nu^+$ denote the resulting
partition.  This  need not be a B-partition, so  take the B-collapse and call this partition $\nu^B$;
this is the largest B-partition (in the partial order on partitions) less that or equal to $\nu^+$.
 Then $d(\caO)$ is parametrized
by the {\em transpose} of $\nu^B$.  (This is automatically a
B-partition  in this case.)

It is now a straightforward matter to check that the condition on the Jordan form of $\caO$
stated in the proposition is equivalent to $\caO^\vee=d(\caO)$ being even.  So these are only 
possible orbits so that $\chi(d(\caO))$ is integral.  
Using the description of the middle element of Jacobson-Morozov triples in the classical cases
(given, for example, in \cite[Chapter 5]{cm}), one checks directly that
$\chi(d(\caO))$ coincides with $\chi'(\caO)$, as claimed.

\end{proof}

\subsection{Counting unipotent representations}
\label{s:sppq-counting}

\begin{theorem}
\label{t:sppq-counting}
Fix a complex nilpotent orbit $\caO$ meeting $\frp$, and write
\[
\caO  \cap \frp = \{ \caO_K^1, \dots, \caO_K^r\}.
\]
Recall the infinitesimal character $\chi'(\caO)$ and primitive ideal $I'(\caO)$
of Section \ref{s:sppq-ic}.  Write $\Unip'(\caO))$ for the
set of irreducible Harish-Chandra modules for $\Sp(p,q)$ annihilated by $I'(\caO)$.  
Then there is a bijection
\begin{align*}
\{ \caO_K^1, \dots, \caO_K^r\} &\lra \Unip'(\caO) \\
\caO_K^i &\lra \pi'(\caO_K^i)
\end{align*}
such that $\caO_K^i$ is the unique dense $K$ orbit in the associated variety of $\pi'(\caO_K^i)$.
\end{theorem}

\begin{proof}
Fix $\caO_K$ and set $\caO = G \cdot \caO_K$.
By \cite[Theorem 5.24]{t:ltc}, the associated variety of any irreducible Harish-Chandra module for $\Sp(p,q)$ with integral infinitesimal character 
is the closure of single nilpotent $K$ orbit on $\frp$.  Moreover \cite[Theorem 5.24]{t:ltc} implies  the set of 
irreducible Harish-Chandra modules, say $\caC(\caO_K)$, with trivial infinitesimal character and  associated variety equal to the
closure of $\caO_K$ is a cell of Harish-Chandra modules.  If we translate the elements of $\caC(\caO_K)$ from trivial infinitesimal
character to $\chi'(\caO_K)$ (without crossing any walls), most elements will die; the representations that do not are irreducible,
have infinitesimal character $\chi'(\caO_K)$,
have associated variety equal to the closure of $\caO_K$, and thus (by Proposition \ref{p:sppq-pi}) annihilator equal to $I'(\caO_K)$.  Finally
 the module structure of $\caC(\caO_K)$ computed in 
 \cite[Theorem 6]{monty} implies that there is exactly one such module.  It is the unique irreducible module
 with associated variety equal to the closure of $\caO_K$ and annihilator $I'(\caO_K)$, namely $\pi'(\caO_K)$.   The theorem follows.
 
 \end{proof}

\begin{remark}
\label{r:sppq-counting}
The proof shows that  $\pi'(\caO_K)$ has another characterization: it is the unique Harish-Chandra module
with infinitesimal character $\chi'(\caO_K)$ and associated variety equal to the closure of $\caO_K$.
\end{remark}

\begin{remark}
\label{r:sppq-mult}
In fact, in Theorem \ref{t:sppq-mult} we will see that the multiplicity of 
$\caO_K$ in the associated cycle of $\pi'(\caO_K)$ is always one. 
\end{remark}

\subsection{Unipotent $A_\frq(\lam)$ modules: all even parts of $\caO_K$ have multiplicity at most 2.}
\label{s:sppq-aql}
As a consequence of the general results of Section \ref{s:coh-ind}, we can describe some unipotent representations for $\Sp(p,q)$ as $\Aql$ modules.  

Recall that the $K$ conjugacy of $\theta$-stable parabolics for $\Sp(p,q)$ are parametrized by 
sequences $(p_0,q_0),(p_1,q_1),\dots, (p_r,q_r)$ with $p = \sum_i p_i$ and
$q= \sum_i q_i$.  If $\frq = \frl \oplus \fru$ is the corresponding parabolic, then the normalizer in $\Sp(p,q)$ of $\frl$ is
\[
\Sp(p_0,q_0) \times \U(p_1,q_1) \times \cdots \times \U(p_r,q_r).
\]
Fix a nilpotent $K$ orbit $\caO_K$ corresponding to a signed tableau $S$ (Section \ref{s:sppq-nilp}) in which all even parts have multiplicity at most 2.   We describe a corresponding
$\theta$-stable parabolic $\frq(\caO_K)$.

First modify $S$ to obtain a new signed tableau $S_1$ as
follows. Let the odd rows of $S_1$ equal the odd rows of $S$. Change each pair of even rows $(m_+,m_-)$ in $S$ to a pair of odd rows $((m+1)_+, (m-1)_-)$ in $S_1$.  Then $S_1$ no longer is of the form to parametrize a
nilpotent orbit in $\Sp(p,q)$,  but $S_1$ is now a tableau in which only odd rows 
appear\footnote{In the definition of $S_1$, for each pair of even rows, we could have made a different choice by changing the pair $(m_+,m_-)$ to  $((m+1)_-, (m-1)_+)$}.  All of
the results below remain valid for any of these possible choices defining $S_1$.  See the remark after Proposition \ref{p:sppq-unip-aql}..

Write $2r+1$ for the largest part of $S_1$.   Set
\begin{align*}
p_r &=  \text{the number of rows of length $2r+1$ beginning with plus};\\
q_r &=  \text{the number of rows of length $2r+1$ beginning with minus};\\
p_{r-1} &=  q_r + \text{the number of rows of length $2r-1$ beginning with minus};\\
q_{r-1} &=  p_r + \text{the number of rows of length $2r-1$ beginning with plus};\\
&\vdots \\
p_0 &= q_1 +  \text{the number of rows of length $1$ beginning with $(-1)^r$};\\
q_0 &= p_1 +  \text{the number of rows of length $1$ beginning with $-(-1)^r$}.
\end{align*}
In other words, $(p_r,q_r)$ records the signs on the ends of the longest rows of $S_1$ or, equivalently (since the rows are odd), the signs at the
beginning of the longest rows of $S_1$.   To obtain $(p_{r-1},q_{r-1})$, remove these the signs $(p_r,q_r)$ from both the beginning and ends of the
longest rows of $S_1$; then $(p_{r-1},q_{r-1})$ records the signs on the ends (equivalently, beginning) of the longest rows of the resulting tableau.  And so on.

Next let
\begin{equation}
\label{e:sppq-q}
\frq(\caO_K) = \frl(\caO_K) \oplus \fru(\caO_K)
\end{equation}
be the parabolic parametrized by the sequence $(p_0,q_0),(p_1,q_1),\dots, (p_r,q_r)$.

Here is the reason for these definitions.

\begin{prop}
\label{p:sppq-birat}
With notation as in the previous paragraph, the dense $K$ orbit
in 
\[
K\cdot (\fru(\caO_K) \cap \frp)
\]
is $\caO_K$.  Moreover,
\[
\dim
\left 
(
G\cdot (\fru(\caO_K) \cap \frp) 
\right )
= 
\dim
\left (
G \cdot (\fru(\caO_K))
\right
).
\]
\end{prop}
\begin{proof}[Sketch]
The first assertion follows from the calculation of $K\cdot(\fru \cap \frp)$ (for general $\fru$) given in \cite[Proposition 5.1]{t:rich}.  Roughly speaking, the
calculation goes as follows.  Let $c(p_i,q_i)$ denote a single signed column with $p_i$ plusses and $q_i$ minuses.  The
signed tableau parametrizing the dense orbit $K$ orbit in $K\cdot(\fru(\caO_K) \cap \frp)$ is obtained by successively adding consecutive signed columns 
\[
c(p_r,q_r), c(p_{r-1},q_{r-1}), \cdots, c(p_1,q_1), c(p_0,q_0), c(p_1,q_1), \cdots, c(p_{r-1},q_{r-1}), c(p_{r},q_{r})
\]
and then collapsing to obtain a signed partition satisfying the conditions of Section \ref{s:sppq-nilp}.  The definitions of $(p_i,q_i)$ given above
are arranged so that the result of adding the consecutive columns together is the tableau $S_1$.  At the same same time, $S_1$ is defined so that its collapse
is $S$, the tableau parametrizing $\caO_K$.  So everything has been engineered so that $\caO_K$ is dense in $K\cdot (\fru(\caO_K) \cap \frp)$, as claimed.

Meanwhile, the general calculation
of the induced orbit dense in $G\cdot \fru$ (for general $\fru$) goes back to Lusztig and is given in \cite[Theorem
7.3.3]{cm}.  So to prove the second assertion of the proposition, we simply have to compare the two calculations.

In more detail, let $c(p_i,q_i)$ denote a single column with $p_i+q_i$ boxes.  
The
diagram parametrizing the dense $G$ orbit in $G\cdot(\fru(\caO_K))$  is obtained by successively adding consecutive columns 
\[
c(p_r+q_r), \cdots, c(p_1+q_1), c(p_0+q_0), c(p_1+q_1), \cdots,c(p_{r}+q_{r})
\]
and then collapsing to obtain a diagram in which every odd part occurs an even number of times.  The definitions of $(p_i+q_i)$ given above
are arranged so that the result of the adding the consecutive columns together is the shape of the tableau $S_1$, which collapses to a diagram
with the shape of $S$.  

Comparing the two calculations, we see that the Jordan form
in both cases is the same.  So the second assertion of the proposition follows.
\end{proof}

\begin{lemma}
\label{l:sppq-lam}
There is a one dimensional $(\frl(O_K), L(\caO_K) \cap K)$ module $\bbC_{\lam(\caO_K)}$ in the weakly fair range
for $\frq(\caO_K)$ such that
$\lam(\caO_K) + \rho$ represents the same infinitesimal character as $\chi'(\caO)$.
\end{lemma}

\begin{proof}
This is clear from the explicit description of $\chi'(\caO)$ given in Section \ref{s:sppq-ic}.
\end{proof}

\begin{prop}
\label{p:sppq-unip-aql}
Fix a complex nilpotent orbit $\caO$ meeting $\frp$ whose Jordan form has all even parts occuring at most twice.  
In the notation of Theorem \ref{t:sppq-counting}, Equation \eqref{e:sppq-q}, and Lemma \ref{l:sppq-lam},
\[
\pi'(\caO_K) = A_{\frq(\caO_K)}(\lam(\caO_K)).
\]
In particular, $\pi'(\caO_K)$ is unitary and $\caO_K$ occurs with multiplicity one in its associated variety.
\end{prop}
\begin{proof}
By Theorem \ref{t:aql}(iii), $A_{\frq(\caO_K)}(\lam(\caO_K))$ has associated variety $K\cdot(\fru(\caO_K) \cap \frp)$; by Proposition \ref{p:sppq-birat},
this contains the dense $K$ orbit $\caO_K$.  By Theorem \ref{t:aql}(i), $A_{\frq(\caO_K)}(\lam(\caO_K))$ has infinitesimal character represented by  $\lam(\caO_K) +\rho$;
by Lemma \ref{l:sppq-lam}, this is the same as the infinitesimal character represented by $\chi'(\caO_K)$.  
The characterization of $\pi'(\caO_K)$ in Remark \ref{r:sppq-counting} now implies that $\pi'(\caO_K) = A_{\frq(\caO_K)}(\lam(\caO_K))$.  The unitarity is
assertion is given by Theorem \ref{t:aql}

Finally, for the multiplicity statement, the component group $A_K(\xi)$ of the centralizer in $K$ of a nilpotent element $\xi \in \frp$ is always trivial; this follows from
direct calculation, along the lines of the centralizer calculations in \cite[Chapter 6]{cm}.  By Theorem \ref{t:aql}(iv), the multiplicity one assertion follows.
\end{proof}

\begin{remark}
\label{r:sppq-unip-aql-other}
Proposition \ref{p:sppq-unip-aql} does not cover all cases in which $\pi(\caO_K)$ is
a weakly fair $\Aql$ module.  (For example, if $G=\Sp(2,2)$ and $\caO_K$ is the unique orbit with Jordan form $2^4$, then the proof of the proposition
also shows that $\pi'(\caO_K)$ is of the form $\Aql$ with $L = \U(2,2)$.)  However
the propositions cover enough cases to allow general unitarity argument in Section
\ref{s:intro2-proof} to proceed,
\end{remark}

\begin{remark}
\label{r:sppq-unip-aql}
As remarked above, we could have made different choices for the tableau $S_1$.  These different choices could have led to different
parabolics $\frq(\caO_K)$, and ostensibly different $A_{\frq(\caO_K)}(\lam(\caO_K))$ modules.  Nonetheless the proof of Proposition \ref{p:sppq-unip-aql} shows that
different choices lead to the {\em same} module.
\end{remark}

\subsection{Proof of Theorem \ref{t:intro2}.}
\label{s:intro2-proof}
For certain kinds of orbits, we have proved that $\pi'(\caO_K)$ is unitary in Section \ref{s:sppq-aql} using cohomological
parabolic induction.  We now formulate an inductive argument using real parabolic induction to reduce the general case to this case
based on the following result on the level of orbits.  

\begin{lemma}
\label{l:sppq-real-ind-orbits}
Suppose $G_\bbR = \Sp(p',q')$ and fix integers $(p,q)$ and $k$ such that $p+k=p'$ and $q+k=q'$.  Fix a
parabolic subgroup $P_\bbR = L_\bbR N_\bbR$ with 
\[
L_\R \simeq \Sp(p,q) \times \GL(k,\bbH).
\]
Fix a real nilpotent orbit $\caO_\bbR$ for $\Sp(p,q)$ parametrized by a signature $(p,q)$ signed tableau $S$ of the kind 
described in Section \ref{s:sppq-nilp}.  Let $S'_1, \dots,S'_r$ obtained from $S$ by increasing $2k$ of the largest
parts of any representative of $S$ (for the equivalence of interchanging rows of equal length)
by two in any way possible subject to the conditions in Section \ref{s:sppq-nilp}; see Example \ref{e:sppq-real-ind-orbits}.

Write $\caO'_{\bbR,1}, \dots \caO'_{\bbR,r}$ for the corresponding nilpotent orbits.  Then
\[
G_\bbR \cdot (\caO_\bbR + \frn_\bbR) = \bigcup_{i=1}^r \ol{\caO'_{\bbR,i}}
\]
\end{lemma}

\begin{proof}
This is an unenlightening exercise with the classification of nilpotent orbits that we omit.
\end{proof}

\begin{example}
\label{e:sppq-real-ind-orbits}
Suppose $G'_\R=\Sp(6,5)$,  $k=1$, $G_\R = \Sp(5,4)$, and $\caO_\bbR$ 
parametrized by the signed tableau $S$ specified by $3_+^23_-^22_-^12_-^+1_+^2$.  There are two possible ways to
increase the first two rows of $S$ by 2: $S_1'=5_+^2 3_-^2 2_+^12_-^+1_+^2$; $S'_2 = 5_-^23_+^22_+^12_-^+1_+^2$.  In this case, 
the corresponding parabolically induced representation has reducible associated variety.

However if
$G'_\R=\Sp(8,7)$,  $k=3$, $G_\R = \Sp(5,4)$, and $\caO_\bbR$ 
parametrized by the signed tableau $S$ specified by $3_+^2 3_-^2 2_-^1 2_-^+ 1_+^2$, there is just one
possible way to increase the first six rows of $S$ by 2: $S_1'=5_+^2 5_-^2 4_-^1 4_-^+1_+^2$. In this case, 
the corresponding parabolically induced representation has irreducible associated variety.
\end{example}

\begin{remark}
\label{r:sppq-real-ind-orbits}
In fact, based on the previous example, we can deduce a general irreducibility result.
Suppose $2r$ is an even part occurring in $S$ with multiplicity $2a$, fix $0\le a' \le a$, and 
let $2b$ be the sum of the multiplicities of all parts strictly greater than $r$.  If we set $k=a'+b$, then there is unique
$S'$ that can be obtained by increasing the largest $a'+b$ parts of $S$ by two.  (In the example in
the previous paragraph, $r=1$, $a=a'=1$, $b=2$, and $k=3$.)  In this case, the corresponding 
the corresponding parabolically induced representation has irreducible associated variety corresponding
to $S'$.

We will use this below as follows.  If $a>1$ and we choose $a'=1$, then comparing $S'$ to $S$, we have:
replaced the occurrence of each even part $2r'$ with $r'>r$ in $S$ by the an occurrence of $2r'+2$ in $S'$;
added the part $2r+2$ with multiplicity two in $S'$;
reduced the multiplicity of the part $2r$ in $S$ by two in $S'$; and kept the even parts of size smaller than $2r$
the same in $S'$ as in $S$.  Repeating
this procedure as necessary, we can therefore arrive at a nilpotent orbit parametrized by a signed tableau in which
each even part occurs with multiplicity at most two (i.e.~the setting of Section \ref{s:sppq-aql}).
\end{remark}

\begin{lemma}
\label{l:sppq-main}
Fix a nilpotent $K$ orbit $\caO_K$ on $\frp$ for $\Sp(p,q)$, and let $\pi'(\caO_K)$ denote the representation
of Theorem \ref{t:sppq-counting}.  Let $S$ be the signed tableau parametrizing $\caO_K$ as
in Section \ref{s:sppq-nilp}, and assume there is an even part $2r$ occuring with nonzero multiplicity in $S$.
Let $2b$ be the sum of the multiplicities of all parts strictly greater than $r$, and set $k = b+1$.
(This is the setting of Remark \ref{r:sppq-real-ind-orbits}.)  Write $\caO_K^\big$ for nilpotent
orbit parametrized by the unique signed tableau obtained by increasing the largest $k$ parts
of $S$ by two (see Remark \ref{r:sppq-real-ind-orbits}).

Fix a real parabolic subgroup
$P_\bbR = L_\bbR N_\bbR$ with 
\[
L_\bbR \simeq \Sp(p,q) \times \GL(k,\bbH).
\]
For $t\in \bbR$, let $\det^t$ denote the $t$-th power of the determinant representation of $\GL(k,\bbH)$ 
and form the induced representations of $G_\bbR^\big = \Sp(p+k,q+k)$
\[
I(t) = \Ind_{(\Sp(p,q) \times \GL(k,\bbH))N_\bbR}^{G_\bbR^\big}
\left (
(\pi'(\caO_K) \boxtimes {\det}^t) \boxtimes 1
\right ).
\]
Then 
\begin{enumerate}
\item[(i)]
For $0 \leq t \leq 1/2$, $I(t)$ is irreducible.  

\item[(ii)]
For $0 \leq t \leq 1/2$, $I(t)$ is unitary if and only if $\pi'(\caO_K)$ is unitary.

\item[(iii)] $I(1/2) = \pi'(\caO_K^\big)$.
\end{enumerate}
\end{lemma}
\begin{proof}
The infinitesimal character character of $I(t)$ is the integral infinitesimal character $\chi'(\caO_K)$ (described Section \ref{s:sppq-ic}) concatenated 
with the infinitesimal character of $\det^t$,
\[
\left (
t + (n-1)/2, t+(n-3)/2, \dots, t-(n-1)/2)
\right ).
\]
This concatenation is non integral for $0 \le t < 1/2$.  The first possible point of reducibility is therefore $t=1/2$, and (i) follows for $0\le t < 1/2$.  (We will return 
to the $t=1/2$ case below.) 

Since $\Sp(p,q)$ is an equal rank group and the infinitesimal character of $\pi(\caO_K)$ is real, $\pi(\caO_K)$ admits a nondegenerate invariant Hermitian form unique up
to scalar.  Since $\det^0$ is trivial and of course unitary,  $\pi(\caO_K) \boxtimes 1$ admits a nondegenerate invariant Hermitian form.  This form induces an invariant
Hermitian form on $I(0)$, and since $I(0)$ is irreducible, the form is nondegenerate and unique up to scalar.  Hence $I(0)$ is unitary if and only if $\pi(\caO_K)$ is unitary.
Since $I(t)$ is irreducible for $0\le t < 1/2$, $I(t)$ is unitary for $0 \le t < 1/2$  if and only if $\pi(\caO_K)$ is unitary.  A standard continuity argument implies that this is also true
at $t=1/2$, and hence (ii) follows.

From the description of the infinitesimal character of $I(t)$ above and the details of Section \ref{s:sppq-ic}, it follows that the infinitesimal character of $I(1/2)$
is $\chi'(\caO_K^\big)$.  By Proposition \ref{p:av-real-ind} (and the discussion after it), Lemma \ref{l:sppq-real-ind-orbits}, and Remark \ref{r:sppq-real-ind-orbits}, it follows that the associated variety of $I(1/2)$ is the closure of $\caO_K^\big$.  By Theorem \ref{t:sppq-counting}, it follows that $I(1/2)$ is a multiple of $\pi'(\caO_K^\big)$.  Hence
(iii) follows once we show $I(1/2)$ is irreducible.

For the irreducibility at $t=1/2$, since we know that $I(1/2)$ is a multiple of $\pi'(\caO_K^\big)$, we need only find a $K$-type in $I(1/2)$ with multiplicity one.  Since $I(1/2)$
and $I(0)$ have the same $K$-types, we can work with the latter module instead, and prove it has a $K$-type with multiplicity one.  

In fact, we sketch the following stronger statement:  Let $\pi$ be {\em any} irreducible representation of $\Sp(p,q)$, let $1_k$ denote the trivial representation of $\GL(k,\bbH)$, and consider the induced representation 
\[
\pi^\big = 
 \Ind_{(\Sp(p,q) \times \GL(k,\bbH))N_\bbR}^{G_\bbR^\big}
\left (
(\pi \boxtimes 1_k) \boxtimes 1
\right ).
\]
Fix a lowest $\Sp(p) \times \Sp(q)$ type $\mu$ of $\pi$; in standard coordinates write the highest weight as
\[
(\mu_1^+, \dots, \mu_p^+; \mu_1^-, \dots, \mu_q^-).
\]
Then we claim that the $\Sp(p+k) \times \Sp(q+k)$ type $\mu^\big$ with highest weight
\begin{equation}
\label{e:mubig}
(\mu_1^+, \dots, \mu_p^+,\overbrace{0,\dots,0}^k; \mu_1^-, \dots, \mu_q^-,\overbrace{0,\dots,0}^k).
\end{equation}
appears with multiplicity one in $\pi^\big$.   A simple Frobenius reciprocity
calculation using the compact picture of $\pi^\big$ shows that $\mu^\big$ appears.  The
remaining task is to see it appears with multiplicity one.

To see this, first embed $\pi$ in a standard module $I_\pi$ (induced from
a cuspidal parabolic subgroup); $I_\pi$ is completely determined by $\mu$ according to the Vogan classification.
Next consider the spherical principal series $I_\circ$ for $\GL(k,\bbH)$ (induced from
 a representative of the unique conjugacy class of
cuspidal parabolic subgroups of $\GL(n,\bbH)$) with the trivial representation $1_k$ as a submodule.  Then
\[
\pi^\big \hookrightarrow 
I^\big:=\Ind_{(\Sp(p,q) \times \GL(k,\bbH))N_\bbR}^{G_\bbR^\big}
\left (
(I_\pi \boxtimes I_\circ) \boxtimes 1
\right ).
\]
By induction in stages, $I^\big$ is now a standard parabolically induced
representation from a cuspidal parabolic subgroup for $\Sp(p+n,q+n)$, and we can 
compute its lowest $K$-types explicitly using the lowest $K$-type formula: we find that the $K$-type $\mu^\big$
of \eqref{e:mubig}
occurs as a lowest $K$-type, and hence has multiplicity one. Since $\pi^\big$ embeds in $I^\big$, we conclude
that this $K$-type has multiplicity {at most} one in $\pi^\big$.  But we have already seen that it
appears with nonzero multiplicity, and hence it appears 
with multiplicity one, as claimed.

Applying this argument to $I(0)$ allows us to conclude that $I(0)$, hence $I(1/2)$, contains a $K$-type
with multiplicity one and is irreducible.  This completes the proof of (iii).
\end{proof}

\begin{theorem}[{Theorem \ref{t:intro2}}]
\label{t:sppq-main}
Fix a complex nilpotent orbit $\caO$ meeting $\frp$.  Then each
representation $\pi'(\caO_K)$ in $\Unip'(\caO)$ (with notation as in Theorem \ref{t:sppq-counting}) is
unitary.
\end{theorem}

\begin{proof}
Fix $\caO_K$ and let $S$ denote the signed tableau parametrizing $\caO_K$ (Section \ref{s:sppq-nilp}).  
If no even part of $S$ exisst with multiplicity strictly greater than two , then $\pi'(\caO_K)$ in unitary by Proposition \ref{p:sppq-unip-aql}, and we are done.
If there is an even part, say $2\ell$ of $S$ with multiplicity greater than two, choose the even part with the largest possible multiplicity.  (This part is not necessarily unique.)  Lemma \ref{l:sppq-main} shows that $\pi'(\caO_K)$ is unitary if and only if
$\pi'(\caO_K^\big)$ is unitary where $\caO_K^\big$ is now parametrized by a new tableau $S'$ with the properties described at the end of Remark \ref{r:sppq-real-ind-orbits}.
As discussed in the remark, by repeatedly applying this procedure, Lemma \ref{l:sppq-main} implies that
the unitarity of $\pi'(\caO_K)$ is equivalent to the unitarity of some $\pi'(\caO_K^\big)$ where no even part of the tableau parametrizing 
$\caO_K^\big$ exist with multiplicity strictly greater than two.  Hence Proposition \ref{p:sppq-unip-aql} now applies to imply unitarity of $\pi'(\caO_K)$.
\end{proof}

\subsection{Proof of Theorem \ref{t:intro1}.}
\label{s:intro1-proof}
By the paragraph after Theorem \ref{t:intro2}, the results of the previous section imply Theorem \ref{t:intro1} for even orbits $\caO^\vee$.  So assume that $\caO^\vee$ is a nilpotent orbit for $\SO(2n+1,\bbC)$
which is not even.  This means that there
are some nonzero parts of even length in the Jordan form of $\caO^\vee$ (and such parts must come in pairs).  List the even length parts as
\[
2k_1 = 2k_1 \geq 2k_2 = 2k_2 \geq \cdots \geq 2k_\ell = 2k_\ell >0.
\]
Set $k=\sum_{i=1}^\ell 2k_i$.
Let $\caO^\vee_0$ denote the nilpotent orbit for $\SO(2n-2k-1,\bbC)$ whose Jordan form is obtained from that of $\caO^\vee$ by removing all
parts of even length; so $\caO^\vee_0$ is now even.  Write $\caO= d(\caO^\vee)$, a nilpotent orbit for $\Sp(2n,\bbC)$, and
$\caO_0= d(\caO^\vee)$, a nilpotent orbit for $\Sp(2n-2k,\bbC)$.

Assume that $p+q = n$ are fixed so that $\caO$ intersects $\frp$ for $\Sp(p,q)$ nontrivially, and
fix an $\Sp(p)\times \Sp(q)$ orbit $\caO_K$ on $\caO \cap \frp$.  Similarly assume that  
$\caO_0$ intersects $\frp$ for $\Sp(p-k,q-k)$ nontrivially, and
fix an $\Sp(p-k)\times \Sp(q-k)$ orbit $\caO_{K,0}$.  We have already constructed the integral
special unipotent representations $\pi'(\caO_{K,0})$ for $\Sp(p-k,q-k)$ and established its unitarity in the previous section.  
Set
\[
L_\R \simeq \Sp(p-k,q-k) \times \GL(k_1,\bbH) \times \cdots \times \GL(k_r,\bbH),
\]
let $1_{k_i}$ denote the trivial 
representation of $\GL(k_i,\bbH)$, and consider the unitarily induced representation
\[
\pi(\caO_K^\vee) = \Ind_{L_\R N_\R}^{\Sp(p,q)} \left (
(\pi'(\caO_{K,0}^\vee) \boxtimes 1_{k_1} \boxtimes \cdots \boxtimes 1_{k_r}) \boxtimes 1 \right ).
\]
By integrality considerations, $\pi(\caO_K^\vee)$ is irreducible and unitary.
So it remains to verify that the unitary representations $\pi(\caO_K^\vee)$ exhaust the
special unipotent representations attached to $\caO^\vee$.

It is easy to check that $\pi(\caO_K^\vee)$ has infinitesimal character $\chi(\caO^\vee)$.
By Proposition \ref{p:av-real-ind} Lemma \ref{l:sppq-real-ind-orbits}, we conclude that
\[
\AV(\pi(\caO_K)) = \ol{\caO_K}.
\]
By Proposition \ref{p:av-unip}, $\pi(\caO_K)$ is indeed special unipotent attached to $\caO^\vee$.  A simple counting argument 
(or else appealing to general endoscopic reduction to the even case considered in \cite[Chapter 26]{abv})
shows that every nonintegral special unipotent representations is of the form $\pi(\caO_K)$.
This completes the proof of Theorem \ref{t:intro1}.

\subsection{Multiplicities in associated cycles}
\label{s:sppq-mult}
We now return to Remark \ref{r:sppq-mult} and prove:

\begin{theorem}
\label{t:sppq-mult}
The multiplicity of $\caO_K$ in the associated cycle of $\pi'(\caO_K)$ is one.
\end{theorem}

\noindent
{\em Sketch.}
For certain orbits, this assertion is already handled by the last statement of Proposition \ref{p:sppq-unip-aql}.  For an orbit $\caO_K$ not covered by the proposition,
we have seen that in Lemma \ref{l:sppq-main} that we may irreducibly induce $\pi'(\caO_K)$ to some $\pi'(\caO^\big_K)$ which is handled
by Proposition \ref{p:sppq-unip-aql}.  

Next we note that the multiplicity of $\caO_K$ in the associated
cycle of $\pi'(\caO_K)$ divides the multiplicity of $\caO^\big_K$ in the associated cycle of $\pi'(\caO^\big_K)$. 
(A much more refined result giving a precise form of the ratio is contained in \cite{barbasch-hc} in the context of asymptotic cycles.  The statement
for associated varieties then follow from the Barbasch-Vogan conjecture.  However, the simpler statement
that multiplicity of $\caO_K$ divides the multiplicity of $\caO^\big_K$ is much easier: since tensoring with finite-dimensional representations commutes
with parabolic induction, a coherent continuation argument shows that the multiplicity polynomial for $\caO_K$ divides the multiplicity polynomial for
 $\caO^\big_K$.)
In any event, since we know the multiplicity of $\caO^\big_K$ in the associated cycle of $\pi'(\caO^\big_K)$ is one
by Proposition \ref{p:sppq-unip-aql}, the multiplicity of $\caO_K$ in the associated cycle of $\pi'(\caO_K)$ is also one.
\qed

\section{details for $\SO^*(2n)$}
\label{s:sostar}
In this section we restrict to $G_\R = \SO^*(2n)$ and use the notation of Section \ref{s:gen-not} in this setting. The results for $\SO^*(2n)$ are entirely parallel to those for $\Sp(p,q)$, and this section should be read by frequently refering
back to the previous section.  Key differences (often related to parity) are emphasized by the use of boldface.  Since the proofs are very similar to those in Section \ref{s:sppq}, we omit them.

\subsection{Nilpotent orbits}
\label{s:sostar-nilp}
The nilpotent orbits of $G$ on $\frg$ are parametrized by partitions of $2n$ in which even parts occur with even multiplicity.
The nilpotent $K$ orbits on $\frp$  are parametrized by signed tableaus of arbitary signature satisfying certain conditions (\cite[Chapter 9]{cm}).
If we write the data of a signed tableau (as in Section \ref{s:sppq-nilp}) as
\[
(m_1)_+^{k_{m_1}^+}(m_1)_-^{k_{m_1}^-}(m_2)_+^{k_{m_2}^+}(m_2)_-^{k_{m_2}^-}\dots
\]
the restrictions are
 \begin{enumerate}
 \item[(i)] For every {\bf even} $m$, $k_m^\vareps$ is even; that is, all even parts beginning with a fixed sign occur an even number of times;
  \item[(i)] For each {\bf odd} part $m$, $k_m^+ = k_m^-$ is even; that is,  the number of rows of a fixed odd length beginning with $+$
  equals the number rows of that length beginning with $-$.
  \end{enumerate}
In particular, once again all parts occur with even multiplicity (or, equivalently, all columns have even length).   
We will sometimes refer to the parts as rows.

\subsection{Infinitesimal Characters}
\label{s:sostar-ic}
Fix a nilpotent orbit $\caO$ for $G$,  consider the corresponding partition of $2n$, and
let $c_1\ge c_2 \ge \cdots$ denote the transposed partition; these are
the columns of $\caO$.  A
necessary condition for $\caO$ to meet $\frp$ is that all
$c_i$ are even, so write $c_i =2d_i$.  To each $d_i$ associate a string of nonnegative integers,
\[
\text{$(1,\dots,d_i)$  if $i$ is {\bf even} and  $(0,\dots,d_{i-1})$ if $i$ is {\bf odd}}.
\]
Let $\chi'(\caO)$ denote the concatenation of these strings  interpreted as an infinitesimal character.

\begin{prop}
\label{p:sostar-pi}
With notation parallel to Proposition \ref{p:sppq-pi}, $I'(\caO)$ is
the unique primitive ideal containing $Z'(\caO)$ whose associated variety is the closure of $\caO$.
\end{prop}

\bigskip
Recall the Spaltenstein dual orbit $d(\caO)$ for $\SO(2n,\C)$ and the infinitesimal character $\chi(d(\caO))$
discussed in the introduction. 

\begin{prop}
\label{p:sostar-ic}
Fix a complex nilpotent orbit $\caO$ for $\SO(2n,\bbC)$ meeting $\frp$ for $\SO^*(2n)$.  Then
\[
\chi(d(\caO)) = \chi'(\caO) \text{ as infinitesimal characters if and only if  $\caO^\vee=d(\caO)$ is even}.
\]
If we write $k$ for the largest part of the Jordan form corresponding  to $\caO$,
 $\caO^\vee$ is even iff every {\bf odd} part less than or equal to $k$ actually appears with nonzero multiplicity
  in the Jordan form of $\caO$.
\end{prop}

\subsection{Counting Unipotent Representations}
The situation is entirely parallel to Section \ref{s:sppq-counting}; the references used there
also apply to $\SO^*(2n)$.
\label{s:sostar-counting}
\begin{theorem}
\label{t:sostar-counting}
Fix a complex nilpotent orbit $\caO$ meeting $\frp$, and write
\[
\caO  \cap \frp = \{ \caO_K^1, \dots, \caO_K^r\}.
\]
Recall the infinitesimal character $\chi'(\caO)$ and primitive ideal $I'(\caO)$
of Section \ref{s:sostar-ic}.  Write $\Unip'(\caO))$ for the
set of irreducible Harish-Chandra modules for $\SO^*(2n)$ annihilated by $I'(\caO)$.  
Then there is a bijection
\begin{align*}
\{ \caO_K^1, \dots, \caO_K^r\} &\longleftrightarrow \Unip'(\caO) \\
\caO_K^i &\longleftrightarrow \pi'(\caO_K^i)
\end{align*}
such that $\caO_K^i$ is the unique dense $K$ orbit in the associated variety of $\pi'(\caO_K^i)$.
\end{theorem}

\subsection{Unipotent $A_\frq(\lam)$ modules: all {\bf odd} parts of $\caO_K$ have multiplicity at most 2.}
\label{s:sostar-aql}
Fix an orbit $\caO_K$ parametrized by a tableau $S$ all of whose {\bf odd} parts have multiplicity at most 2.

First modify $S$ to obtain a new signed tableau $S_1$ as follows.
Let the {\bf even} rows of $S_1$ equal the {\bf even} rows of $S$, but change each pair of {\bf odd} rows $(m_+,m_-)$ in $S$ to a pair of {\bf even}  rows $((m+1)_+, (m-1)_-)$ in $S_1$.  So $S_1$ is now a tableau in which only {\bf even} rows 
appear.  (Again choices are made in the definition of $S_1$, but these ultimately do not matter.)

Write $2r+2$ for the largest part of $S_1$.   Set
\begin{align*}
p_r &=  \text{the number of rows of length $2r+2$ beginning with plus};\\
q_r &=  \text{the number of rows of length $2r+2$ beginning with minus};\\
p_{r-1} &=  q_r + \text{the number of rows of length $2r$ beginning with minus};\\
q_{r-1} &=  p_r + \text{the number of rows of length $2r$ beginning with plus};\\
&\vdots \\
p_0 &= q_1 +  \text{the number of rows of length $2$ beginning with $(-1)^r$};\\
q_0 &= p_1 +  \text{the number of rows of length $2$ beginning with $-(-1)^r$}.
\end{align*}

The sequence $(p_0,q_0),(p_1,q_1), \dots, (p_r,q_r)$
parametrizes a $\theta$-stable parabolic subalgebra
$\frq(\caO_K) = \frl(O_K) \oplus \fru(\caO_K)$ with corresponding Levi factor
\[
\SO^*(2p_0) \times \U(p_1,q_1) \times \cdots \times \U(p_r,q_r).
\]
The properties of Proposition \ref{p:sppq-birat} carry over as follows.
 This time, one uses \cite[Proposition 6.1]{t:rich} to compute the signed
 tableau for the dense $K$ orbit in $K\cdot(\fru(\caO_K) \cap \frp)$; it
 is obtained by successively adding consecutive signed columns 
\[
c(p_r,q_r), c(p_{r-1},q_{r-1}), \cdots, c(p_1,q_1), c(p_0,q_0), c(q_1,p_1), \cdots, c(q_{r-1},p_{r-1}), c(q_{r},p_{r})
\]
and then collapsing (this time for Type D) to obtain a signed partition satisfying the conditions of Section \ref{s:sppq-nilp}.   Everything has once again been engineered so that $\caO_K$ is dense in $K\cdot (\fru(\caO_K) \cap \frp)$.  Comparing with
the calculation of $G\cdot \fru(\caO_K)$ yields the conclusion of  Proposition \ref{p:sppq-birat}  in the setting of $\SO^*(2n)$.

Next, the analog of Lemma \ref{l:sppq-lam} provides
a character $\lam(\caO_K)$.  The proof of Proposition 
\ref{p:sppq-unip-aql} then applies to give:
\begin{prop}
\label{p:sostar-unip-aql}
Fix a complex nilpotent orbit $\caO$ meeting $\frp$ whose Jordan form has all {\bf odd} parts occurring at most twice.  
In the notation of Theorem \ref{t:sostar-counting}, 
\[
\pi'(\caO_K) = A_{\frq(\caO_K)}(\lam(\caO_K)).
\]
In particular, $\pi'(\caO_K)$ is unitary and $\caO_K$ occurs with multiplicity one in its associated variety.  
\end{prop}

Once again the proposition does not cover all $\pi'(\caO_K)$ which are weakly
fair $\Aql$ modules.  (The trivial representation is such an example.)

\subsection{Proof of Theorem \ref{t:intro2}.}
\label{s:sostar-proof-unitary'}
One proves that $\pi'(\caO_K)$, for general $\caO_K$, is unitary by the exactly the same argument as in Section \ref{s:intro2-proof}.  By using carefully chosen parabolic inductions (to successively reduce the multiplicity of odd parts in the partition parametrizing $\caO_K$), one
finds that $\pi'(\caO_K)$ is unitary if and only if $\pi'(\caO_K^\dagger)$ is unitary, where  $\caO^\dagger_K$ is now an orbit for a larger group 
of the form treated in the previous section.  In particular, the statement of the
key computation of Lemma \ref{l:sppq-real-ind-orbits} applies without change; and the deformation argument of Lemma \ref{l:sppq-main} carries over in a similar fashion.

\subsection{Proof of Theorem \ref{t:intro1}.}
\label{s:sostar-proof-unitary}
Noneven orbits are treated as in Section \ref{s:intro1-proof}: the corresponding special unipotent representations are irreducibly parabolically induced from an integral special unipotent
representation whose unitarity is handled by the previous section.

\subsection{Multiplicities in associated cycles}
\label{s:sostar-mult}Finally, the analog of Theorem \ref{t:sppq-mult} is proved in the same way.


\begin{thebibliography}{ABHW}

\bibitem[ABV]{abv} Adams, J., D. Barbasch, and D. A. Vogan, Jr., {\em
The Langlands Classification and Irreducible Characters for Real
Reductive Groups}, Progress in Math, Birkh\"auser (Boston), {\bf 104}(1992).

\bibitem[Ba1]{barbasch-icm}
D.~Barbasch,
Unipotent representations for real reductive groups,
in
{\em Proceedings of the International Congress of Mathematicians, Vol. I, II (Kyoto, 1990)}, 769--777.

\bibitem[Ba2]{barbasch-hc}
D.~Barbasch,
Orbital integrals of nilpotent orbits, 
Proc. Sympos. Pure Math., {\bf 68} (2000), 97?110. 


\bibitem[BaBo]{babo}
D.~Barbasch, M.~Bozicevi\'{c},
The associated variety of an induced representation,
{\em Proc.~Amer.~Math.~Soc.}, {\bf 127} (1999), no.~1, 279--288. 


\bibitem[BV1]{bv:class} D.~Barbasch and D.~A.~Vogan, Jr., 
Primitive ideals and orbital integrals in complex classical groups,
{\em Math.~Ann.}, {\bf 259} (1982), no.~2, 153--199.

\bibitem[BV2]{bv2} D.~Barbasch and D.~A.~Vogan, Jr., 
Unipotent representations of complex semisimple groups,
{\em Ann.~of Math.~(2)}~{\bf 121} (1985),  no.~1, 41--110.



\bibitem[BoJ]{borho-jantzen}
W.~Borho, J.~C.~Jantzen,
\"{U}ber primitive Ideale in der Einh\"{u}llenden einer halbeinfachen Lie-Algebra,
{\em Invent.~Math.}, {\bf 39} (1977), no.~1, 1--53.


\bibitem[CM]{cm}
D.~H.~Collingwood and W.~M.~McGovern, {\em Nilpotent orbits in
semisimple Lie algebras}, Chapman and Hall (London), 1994.



\bibitem[KnV]{kv} A. Knapp and D.~A.~Vogan, Jr., \textit{Cohomological Induction and
Unitary Representations}, Princeton University Press, Princeton, NJ, 1995.

\bibitem[MSZ]{msz}
J.~Ma, B.~Sun, C.-B.~Zhu,
Unipotent representations of real classical groups,
{arXiv:1712.05552}.


\bibitem[Mc1]{monty}
W.~M.~McGovern, Cells of Harish-Chandra modules for
real classical groups, {\em  Amer.~J.~Math}~{\bf  120} (1998),
  no. 1, 211--228. 


\bibitem[PT]{pt} A.~Paul, P.~E.~Trapa, One-dimensional representations
of $\U(p,q)$ and the Howe correspondence, {\em J.~Funct.~Anal.} {\bf
195} (2002), no.~1, 129--166.


\bibitem[T1]{t:rich} 
P.~E.~Trapa,
Richardson orbits for real groups,
{\em J.~Algebra}, {\bf 286} (2005), 361--385.

\bibitem[T2]{t:ltc}
P.~E.~Trapa,
Leading term cycles and partial orders on components of the Springer fiber,
{\em Compos. Math.}, {\bf 143} (2007), no.~2, 515--540. 

\bibitem[V1]{v:unit} 
 D.~A.~Vogan, Jr.,
  {Unitarizability of certain series of 
representations}, {\em Ann.~Math.}~{\bf 120}(1984), 141--187. 

\bibitem[V2]{v:ds}
 D.~A.~Vogan, Jr.,
Irreducibility of discrete series representations for semisimple symmetric spaces,
in {\em Representations of Lie groups, Kyoto, Hiroshima, 1986}, 191--221, 
Adv.~Stud.~Pure Math., {\bf 14} (1988), Academic Press (Boston, MA).
  
\bibitem[V3]{v:av}
 D.~A.~Vogan, Jr.,
  {Associated varieties and unipotent representations},
{in} {\em Harmonic analysis on reductive groups (Brunswick, ME,
1989)}, Progress in Math.~{\bf 101}(1991), Birkh\"auser (Boston),
315--388.



\end{thebibliography}
\end{document}